\pdfoutput=1 
\documentclass[11pt]{amsart}
\usepackage[utf8]{inputenc}
\usepackage{amssymb,amsmath,amsthm,amsfonts,enumerate,enumitem,colonequals,tikz-cd,microtype, quiver}
\usepackage[cal=euler,bfcal,bb=px,bfbb]{mathalpha}
\usepackage[top=3.75cm, bottom=3cm, left=3.5cm, right=3.5cm]{geometry}
\usepackage[all,cmtip]{xy}\xyoption{dvips}

\makeatletter
\@namedef{subjclassname@2020}{\textup{2020} Mathematics Subject Classification}
\makeatother

\usepackage{xcolor}
\colorlet{darkblue}{blue!55!black}
\colorlet{darkcyan}{cyan!50!black}
\colorlet{darkgreen}{green!60!black}

\usepackage{hyperref}
\hypersetup{
    colorlinks=true,
    linkcolor= darkblue,
    urlcolor= darkcyan,
    citecolor= darkgreen,
}

\def\eqref#1{\textcolor{darkblue}{(\ref{#1})}}

\PassOptionsToPackage{hyphens}{url}
\usepackage{hyperref}

\usepackage[nameinlink]{cleveref} 
\Crefformat{section}{#2\S#1#3} 
\Crefmultiformat{section}{#2\S\S#1#3}{ and~#2#1#3}{, #2#1#3}{, and~#2#1#3}
\crefname{diagram}{diagram}{diagrams}
\Crefname{diagram}{Diagram}{Diagrams}
\creflabelformat{diagram}{#2(\textup{#1})#3}

\DeclareFontFamily{OT1}{pzc}{}
\DeclareFontShape{OT1}{pzc}{m}{it}{<-> s * [1.100] pzcmi7t}{}
\DeclareMathAlphabet{\mathchanc}{OT1}{pzc}{m}{it}

\definecolor{darkblue}{RGB}{10,92,153} 
\newcommand{\Irr}{{\mathchanc{Irr}}}
\newcommand{\kdot}{{{\,\begin{picture}(1,1)(-1,-2)\circle*{2}\end{picture}\,}}}
\newcommand{\sR}{{\mathchanc{R}\!}}  
\newcommand{\sHom}[0]{{\mathchanc{Hom}}}

\usepackage[pagewise]{lineno}
\let\oldequation\equation
\let\oldendequation\endequation
\renewenvironment{equation}{\linenomathNonumbers\oldequation}{\oldendequation\endlinenomath}
\expandafter\let\expandafter\oldequationstar\csname equation*\endcsname
\expandafter\let\expandafter\oldendequationstar\csname endequation*\endcsname
\renewenvironment{equation*}{\linenomathNonumbers\oldequationstar}{\oldendequationstar\endlinenomath}
\let\oldalign\align
\let\oldendalign\endalign
\renewenvironment{align}{\linenomathNonumbers\oldalign}{\oldendalign\endlinenomath}
\expandafter\let\expandafter\oldalignstar\csname align*\endcsname
\expandafter\let\expandafter\oldendalignstar\csname endalign*\endcsname
\renewenvironment{align*}{\linenomathNonumbers\oldalignstar}{\oldendalignstar\endlinenomath}


\newcounter{intro}
\newcounter{HypCounter}

\newtheorem{introthm}[intro]{Theorem}

\newtheorem{introcor}[intro]{Corollary}

\theoremstyle{plain}
\newtheorem{theorem}{Theorem}[section]
\newtheorem{lemma}[theorem]{Lemma}
\newtheorem{corollary}[theorem]{Corollary}

\theoremstyle{definition}

\newtheorem*{hypothesis*}{Hypothesis}

\newtheorem*{ack}{Acknowledgements}

\numberwithin{equation}{section}
\numberwithin{theorem}{section}

\title
{Simple criteria for higher rational singularities}

\author[S.~J.~Kov\'{a}cs]{S\'{a}ndor J Kov\'{a}cs}
\address{SJK: 
Department of Mathematics,
University of Washington, 
Seattle, WA 98195,
U.S.A.}
\email{skovacs@uw.edu}

\author[P.~Lank]{Pat Lank}
\address{PL: 
Dipartimento di Matematica “F. Enriques”, Universit\`{a} degli Studi di Milano, Via Cesare
Saldini 50, 20133 Milano, Italy}
\email{plankmathematics@gmail.com}

\author[S.~Venkatesh]{Sridhar Venkatesh}
\address{SV: 
Department of Mathematics,
University of Michigan, 
Ann Arbor, MI 48109,
U.S.A.}
\email{srivenk@umich.edu}

\date{\today}

\keywords{Higher singularities, derived categories, splitting criteria, rational singularities, Du~Bois singularities}

\subjclass[2020]{14B05 (primary), 32S35, 14A30, 14E15, 13D09}

\begin{document}
    
\begin{abstract}
  This work establishes simple criteria for detecting higher rational singularities
  via the intersection Du~Bois complex and the irrationality complex of a normal
  variety over the complex numbers.
\end{abstract}

\maketitle


\section{Introduction}
\label{sec:intro}
\noindent
Derived categories have, and continue to play an active role in birational geometry.
They provide a useful tool to study singularities, in particular rational
\cite{Artin:1966, Lipman:1969} and Du~Bois singularities \cite{DuBois:1981,
  Steenbrink:1981, Steenbrink:1983}. The fact that Kawamata log terminal
singularities are rational \cite{Elkik81} has had an enormous effect in many
applications. 
Log canonical singularities, which are extremely important in moduli theory are not
rational, but they are Du~Bois \cite{Kollar/Kovacs:2010}. This latter class should be
thought of as a generalization of rational singularities cf.~\cite{Kovacs:1999}.

The theory of mixed Hodge modules, developed by Saito \cite{Saito:1989, Saito:1990},
is a vast generalization of the classical theory of (mixed) Hodge structures. There
are deep applications in singularity theory, algebraic geometry, and mirror symmetry
(e.g.\ \cite{Katzarkov/Kontsevich/Pantev:2008}). Lately, ideas from mixed Hodge
theory has led to \textit{higher analogs} of rational and Du~Bois
singularities. Initial work focused on local complete intersections to explore higher
versions of these singularities \cite{Mustata/Olano/Popa/Witaszek:2023,
  Jung/Kim/Saito/Yoon:2022, Mustata/Popa:2022, Friedman/Laza:2024,
  Friedman/Laza:2024b, Mustata/Popa:2025}. Subsequently, 
the investigation has been extended to arbitrary varieties
\cite{Shen/Venkatesh/Vo:2023, Tighe:2024, Popa/Shen/DucVo:2024, Park/Popa:2024,
  Dirks/Olano/Raychaudhury:2025, Kovacs:2025}.

These new classes of singularities are defined by vanishing conditions on cohomology
sheaves of specific graded components of bounded complexes derived from mixed Hodge
modules.  In addition to the Deligne-Du~Bois complex $\underline{\Omega}^\kdot_X$,
key examples include the intersection Du~Bois complex
$\mathcal{I} \underline{\Omega}^\kdot_X$ and the irrationality complex $\Irr^\kdot_X$
\cite{Popa/Shen/DucVo:2024,Kovacs:2025}. Their associated graded complexes,
$\underline{\Omega}^p_X, \mathcal{I}\underline{\Omega}^p_X, \Irr^p_X$ lie in the
bounded derived category of coherent sheaves, $D^b_{\operatorname{coh}}(X)$, for
complex varieties. See \Cref{sec:prelim_singularities} for details.

An important aspect of extending these notions to non-lci singularities is that the
original definition in the lci case is too restrictive in general
cf.~\cite{Shen/Venkatesh/Vo:2023, Tighe:2024}. This led to several variants following
two main motivating principles.

The more obvious principle is to make the definition less restrictive in general
while maintaining the class of singularities it defines in the lci case. This is
achieved by the notions of \emph{$m$-rational} and \emph{$m$-Du~Bois singularities},
which we will review in a moment.

Another motivating principle is to identify and generalize the essential property of
the existing definitions.
Several arguments, already in the `classic' rational and Du~Bois case, only require
the vanishing of higher cohomologies. Arguably this is the most important property of
these singularties and their higher analogs. This leads to the notions of
\emph{pre-$m$-rational} and \emph{pre-$m$-Du~Bois singularities}. Their definition is
essentially that the relevant complexes have a single non-zero cohomology
sheaves. For the precise definition see \Cref{sec:prelim_singularities}.

After this the definition of \emph{$m$-rational} and \emph{$m$-Du~Bois singularities}
in the non-lci case should be based on the definitions of \emph{pre-$m$-rational} and
\emph{pre-$m$-Du~Bois singularities} with some additional assumptions about that
single non-zero cohomology sheaf that the above definition dictates. Unfortunately,
the lci case is even more special and in order to design a general definition that
agrees with the original definition in the lci case one needs to add a codimension
restriction on the singular locus. This might feel somewhat unnatural, or one may
argue that it is a sort of (``super'')-normality condition. In this article we will
not use these notions, so for the precise definitions the reader is referred to
\cite[Defs.~1.2,1.3]{Shen/Venkatesh/Vo:2023}.

All of these notions are relatively new, but it already seems to becoming clear that
the really interesting notions are the {pre-$m$-rational} and {pre-$m$-Du~Bois
  singularities}. Even if one is primarily interested in the $m$-rational and
$m$-Du~Bois cases, the difficult part is likely to deal with the ``pre'' case and the
additional conditions should be relatively easy to handle. Accordingly, in this
article we concentrate on the ``pre'' case.

An important aspect of better understanding a class of singularities is to find
useful criteria that identify these classes.
This has been done for instance in the cases of rational 
\cite{Kovacs:2000,Bhatt:2012,Murayama:2024,Lank:2025} and Du~Bois singularities
\cite{Schwede:2007,KS13}.

As a next step, it is natural to aim to find similar criteria for the singularities
we have been discussing.  Recently \cite{Kovacs:2025} established detection criteria
for pre-$m$-Du~Bois singularities. Here we show analogous results for
pre-$m$-rational singularities. 

Our main result is the following.
\begin{introthm}
  [= \Cref{thm:splitting_via_irr_complex}]
  \label{introthm:kovacs_splitting_pre_k_rational}
  Let $X$ be a normal variety over $\mathbb{C}$. Then the following are equivalent
  for any $m\geq 0$:
  \begin{enumerate}
  \item $X$ has pre-$m$-rational singularities,
  \item
    $\mathcal{H}^0(\mathcal{I}\underline{\Omega}^p_X) \xrightarrow{ntrl.}
    \mathcal{I}\underline{\Omega}^p_X$ has a left inverse for each $0\leq p \leq m$,
    and
  \item $\widetilde{\Irr}^p_X \xrightarrow{ntrl.} \Irr^p_X$ has a left inverse for
    each $0\leq p \leq m$.
  \end{enumerate}
\end{introthm}

\noindent
\Cref{introthm:kovacs_splitting_pre_k_rational} is a genuine strengthening of
\cite{Kovacs:2000,Bhatt:2012} to the realm of higher rational singularities.
It is reasonable to expect that this criterion will be similarly useful.
To demonstrate the power of this criterion we consider finite morphisms of normal
varieties.

\begin{introcor}
  [= \Cref{cor:finite_cover_pre_m_rational}]
  \label{introcor:finite_cover_pre_m_rational}
  Let $f\colon Y \to X$ be a finite surjective morphism of normal varieties. If $Y$
  has pre-$m$-rational singularities, then so does $X$.
\end{introcor}

\noindent
The result above is analogous to the one known for rational singularities.
In that case it is a direct consequence of (the relevant analog of)
\Cref{introthm:kovacs_splitting_pre_k_rational} due to the fact that the trace map
splits the natural map $\mathcal O_X\to f_*\mathcal O_Y$. The proof is less
straightforward in the higher rational case.


Dirks, Olano, and Raychaudhury proved a similar result for finite group quotients
\cite[Corollary 6.6]{Dirks/Olano/Raychaudhury:2025} albeit with a markedly different
proof. For higher Du~Bois singularities, Shen, Venkatesh, and Vo \cite[Proposition
4.2]{Shen/Venkatesh/Vo:2023} established a variant of \cite[Corollary
2.5]{Kovacs:1999}. After finishing this article we learned that Hyunsuk Kim obtained
the same result as in \Cref{cor:finite_cover_pre_m_rational} using a trace morphism
on the Du Bois complex \cite{Kim:2025}, and that Duc Vo also obtained similar results
to some that appear here in his thesis \cite{DucVo:2025}.

\begin{ack}
  S\'andor Kov\'acs was supported in part by NSF Grant DMS-2100389.  Pat Lank was
  supported under the ERC Advanced Grant 101095900-TriCatApp. The authors would like
  to thank J\'anos Koll\'ar, Mircea Musta\c{t}\u{a}, Wanchun Shen, Duc Vo for helpful
  discussions and useful comments and suggestions.
\end{ack}

\section{Preliminaries}
\label{sec:prelim}

We briefly cover the notions of singularities and their constructions related to our
work. Let $X$ be a variety over $\mathbb{C}$ (or an algebraically closed field of
characteristic zero), i.e., a separated integral scheme of finite type over the base
field. %
Recall that a proper birational morphism
$f\colon \widetilde{X} \to X$ from a smooth variety is said to
be a \emph{strong log resolution} with exceptional divisor $E$ if it is an
isomorphism away from the singular locus $X_{\operatorname{sing}}$, and
$E := f^{-1}(X_{\operatorname{sing}})$ with its reduced structure is a simple normal
crossing divisor.

\subsection{Deligne-Du~Bois complexes}
\label{sec:prelim_DB}
Extending Deligne's ideas, Du~Bois \cite{DuBois:1981} constructed the so called
\emph{Deligne-Du~Bois complex}, which behaves similarly to the de~Rham complex in the
smooth case. In fact, it is isomorphic to the de~Rham complex on smooth varieties.
We will briefly review the necessary background, but the reader is encouraged to
peruse \cite{DuBois:1981}, \cite{Guillen/NavarroAznar/Pascual-Gainza/Puerta:1988},
and \cite[\S 7.3]{Peters/Steenbrink:2008} for more details.

Let $\epsilon \colon X_\kdot \to X$ be a cubical hyperresolution of $X$. The
\emph{$p$-th \mbox{(graded Deligne-)}Du~Bois complex} of $X$, denoted by
$\underline{\Omega}^p_X$, is the $p$-th graded component of the filtered
Deligne-Du~Bois complex, $\underline{\Omega}^\kdot_X$.  It was shown in
\cite[4.4]{DuBois:1981} that $\underline{\Omega}^p_X\in D^b_{\operatorname{coh}}(X)$.

The object $\underline{\Omega}^\kdot_X$ can be represented (i.e., up to isomorphism
in the derived category) by the associated total complex of a double complex
$\mathcal{D}^{\bullet , \bullet}$ whose components are of the form
$(\epsilon_s)_\ast \mathcal{A}^{t,s}_{X_s}$
where $\mathcal{A}^{p,q}_{X_i}$ are the sheaves of smooth $(p,q)$-forms on $X_i$
(appearing in $X_{\kdot}$) cf.~\cite[2.4]{Steenbrink85}.  The rows for
$\mathcal{D}^{\bullet , \bullet}$ represent the objects
$\sR (\epsilon_i)_\ast \Omega^p_{X_i}\in D^b_{\operatorname{coh}}(X)$.
By \cite[Lemma 2.1]{Shen/Venkatesh/Vo:2023}, $\underline{\Omega}^p_X$ is obtained
from taking successive cones of morphisms
$\sR (\epsilon_i)_\ast \Omega^p_{X_i} \to \sR (\epsilon_{i+1})_\ast
\Omega^p_{X_{i+1}}$ arising from $\mathcal{D}^{\bullet , \bullet}$. This provides an
alternative proof of the fact that
$\underline{\Omega}^p_X\in D^b_{\operatorname{coh}}(X)$.

\subsection{Mixed Hodge modules}
\label{sec:prelim_MHM}
Mixed Hodge modules
were introduced in \cite{Saito:1989,Saito:1990}. See \cite[\S
14]{Peters/Steenbrink:2008} for a modern treatment. Denote the derived category of
bounded complexes of mixed Hodge modules on $X$ by $D^b(\operatorname{MHM}(X))$.
This category admits a \emph{duality functor}. Specifically, by abuse of notation, it
may be identified with the duality functor $\mathbb{D}_X(-)$ on the derived category
of filtered $D$-modules underlying mixed Hodge modules \cite[\S
2.4]{Saito:1988}. There are the \emph{($p$-th) graded de Rham functors} (which are
exact functors):
\begin{displaymath}
  \operatorname{gr}^F_p \operatorname{DR}_X \colon D^b(\operatorname{MHM}(X)) \to
  D^b_{\operatorname{coh}}(X). 
\end{displaymath}
The duality functor $\mathbb{D}_X(-)$ commutes with $\operatorname{gr}^F_p\operatorname{DR}_X$; in
particular (see \cite[\S 2.4]{Saito:1988}),
\begin{displaymath}
  \sR\sHom  ( \operatorname{gr}^F_p
  \operatorname{DR}_X (E) , \omega^\kdot_X)[- \dim X] \simeq \operatorname{gr}^F_{-p}
  \operatorname{DR}_X  \mathbb{D}_X(E)[- \dim X]. 
\end{displaymath}
%
There are two objects in $D^b(\operatorname{MHM}(X))$ central to the definitions of
singularities we study, namely, $\mathbb{Q}^H_X [\dim X]$ and the intersection
complex $\mathcal{IC}_X^H$, see \cite[\S 4.5]{Saito:1990}. We remark that our notation is different from loc.\ cit.\ which uses the notation $\mathcal{IC}_X \mathbb{Q}^H$. There is a
composition
\begin{equation}
  \label{eq:intersection_composition}
  \gamma_X \colon \mathbb{Q}^H_X [\dim X] \to \mathcal{H}^0 (\mathbb{Q}^H_X [\dim X])
  \to \mathcal{IC}_X^H \simeq \operatorname{gr}^W_{\dim X} (\mathcal{H}^0
  (\mathbb{Q}^H_X [\dim X])) 
\end{equation}
where $\operatorname{gr}^W$ refers to a graded piece with respect to the weight filtration (see loc.\ cit.). By
\cite[Theorem 4.2]{Saito:2000}, there is an isomorphism
\begin{displaymath}
  \underline{\Omega}^p_X \simeq \operatorname{gr}^F_{-p} \operatorname{DR}_X
  (\mathbb{Q}^H_X [\dim X])[p- \dim X] \simeq \operatorname{gr}^F_{-p}
  \operatorname{DR}_X (\mathbb{Q}^H_X [p]). 
\end{displaymath}
This leads to the \emph{intersection Du~Bois complexes} (see \cite[\S
9]{Popa/Shen/DucVo:2024}),
\begin{displaymath}
  \mathcal{I} \underline{\Omega}^p_X := \operatorname{gr}^F_{-p} \operatorname{DR}_X
  (\mathcal{IC}_X^H [p- \dim X]). 
\end{displaymath}
Taking the composition of $\gamma_X$ with its dual gives the following morphisms:
\begin{displaymath}
  \mathbb{Q}^H_X [\dim X] \to  \mathcal{IC}_X^H \to (\mathbb{D}_X
  (\mathbb{Q}^H_X [\dim X]))(-\dim X),
\end{displaymath}
where we use the self-duality relation
$\mathbb{D}_X (\mathcal{IC}_X^H)\simeq \mathcal{IC}_X^H (\dim X)$. Here, $(t)$ is the
$t$-th Tate twist of $M$ (at the level of filtered $D$-modules this is simply a shift
down by $t$, that is, $F_{\kdot} (M)(t) = F_{\kdot-t} M$), cf.~\cite[Eq.\
4.5.13]{Saito:1990}. Finally, applying the functors
$\operatorname{gr}^F_{-p} \operatorname{DR}_X$ give us the following morphisms in
$D^b_{\operatorname{coh}}(X)$:
\begin{equation}
  \label{eq:intro_mhm_composition}
  \underline{\Omega}^p_X \to \mathcal{I} \Omega^p_X \to \sR\sHom (
  \underline{\Omega}^{\dim X-p}_X ,   \omega^\kdot_X)[-\dim X]. 
\end{equation}
To simplify notation, we follow \cite[Eq.\ (3.K.1)]{Kovacs:2025} and use
\begin{equation}
  \label{eq:intro_irrationality_complex_pth}
  \Irr^p_X\colon=\sR\sHom( \underline{\Omega}^{\dim X-p}_X ,
  \omega^\kdot_X)[-\dim X]
\end{equation}
as well as, $\widetilde{\Irr}^p_X := \mathcal{H}^0 (\Irr^p_X)$. Here, $\Irr^p_X$ is
the shifted associated graded quotient of the \emph{irrationality complex of $X$},
which is denoted by $\Irr^\kdot_X$ (see \cite[\S 3.K]{Kovacs:2025}). Up to
isomorphism of filtered complexes on $X$, we have
\begin{displaymath}
  \Irr^{\kdot}_X \simeq \sR\sHom(\sR  \epsilon_\ast
  \underline{\Omega}^\kdot_{X_\kdot}, \omega^\kdot_X ) [-2\dim X]. 
\end{displaymath}
Note that the $2\dim X$ shift is necessary, because of taking the dual and the
inherent shift in $\omega^\kdot_x$. For a longer explanation see
\cite[Remark~3.12]{Kovacs:2025}.

\subsection{Singularities}
\label{sec:prelim_singularities}
Next we consider the singularities we are interested in.
Following \eqref{eq:intro_mhm_composition} and
\eqref{eq:intro_irrationality_complex_pth}, there exist a composition of natural
morphisms given by
\begin{equation}
  \label{eq:1}
  \underline{\Omega}^p_X \to \mathcal{I} \Omega^p_X \to \sR\sHom(
  \underline{\Omega}^{\dim X-p}_X , \omega^\kdot_X)[-\dim X] \simeq \Irr^p_X. 
\end{equation}
For $m\geq 0$, we say that $X$ has
\begin{itemize}
\item \emph{pre-$m$-rational singularities} if $\mathcal{H}^j (\Irr^p_X) = 0$ for all
  $0\leq p \leq m$ and $j>0$,
\item \emph{pre-$m$-Du~Bois singularities} if
  $\mathcal{H}^j (\underline{\Omega}^p_X) = 0$ for all $0\leq p \leq m$ and $j>0$,
\end{itemize}
These (and their variants) were first studied in \cite[Definition
1.1]{Shen/Venkatesh/Vo:2023}, with further variants appearing in \cite[Definition
4.2]{Kovacs:2025}. Observe that we say `pre-$m$-' as opposed to `pre-$k$-' like the
case of loc.\ cit. This follows the convention of \cite[\S 4]{Kovacs:2025} as
`$k$-rational' is often used to refer to points defined over a base field
$k$. Moreover, our definition for pre-$m$-rational singularities is equivalent to
that of \cite[Definition 1.1]{Shen/Venkatesh/Vo:2023} by using the irrationality
complex of $X$ (see \eqref{eq:intro_irrationality_complex_pth}). An important fact is
that a normal variety over $\mathbb{C}$ having pre-$m$-rational singularities implies
that it has pre-$m$-Du~Bois singularities. This is \cite[Theorem
B]{Shen/Venkatesh/Vo:2023} (see also \cite[Theorem 1.4]{Kovacs:2025}).

\section{Main results}
\label{sec:main_results}

We start with some simple results. The following lemma is likely well known, but we
include it for the sake of being self-contained.

\begin{lemma}[Kebekus--Schnell]
  \label{lem:Kebekus_Schnell}
  For any variety $X$ over $\mathbb{C}$ with rational singularities, if
  $0\leq p \leq \dim X$, then
  $\mathcal{H}^0(\underline{\Omega}^p_X) \xrightarrow{ntrl.}
  \mathcal{H}^0(\mathcal{I} \underline{\Omega}^p_X)$ and
  $\mathcal{H}^0(\mathcal{I} \underline{\Omega}^p_X)\xrightarrow{ntrl.}
  \widetilde{\Irr}^p_X$ are isomorphisms, and each of these sheaves is isomorphic to
  the reflexive differentials $\Omega^{[p]}_X$.
\end{lemma}

\begin{proof}
  $\mathcal{H}^0(\underline{\Omega}^p_X) \xrightarrow{ntrl.} \widetilde{\Irr}^p_X$ is
  an isomorphism and both of these sheaves are isomorphic to $\Omega^{[p]}_X$ by
  \cite[Corollary 3.18]{Kovacs:2025}.  This isomorphism factors through
  $\mathcal{H}^0(\mathcal{I} \underline{\Omega}^p_X)$ by \eqref{eq:1}, and
  $\mathcal{H}^0 (\mathcal{I} \underline{\Omega}^p_X)$ is torsion-free by
  \cite[(8.4.1) \& Prop.~8.1]{Kebekus/Schnell:2021}, which implies the desired
  statement.
\end{proof}

Next we show that pre-$m$-rational singularities can be defined in terms of the
intersection Du Bois complexes, a fact that is essentially contained in
\cite{Popa/Shen/DucVo:2024} and \cite{Dirks/Olano/Raychaudhury:2025}.

\begin{lemma}
    \label{lem:pre_m_rational_via_intersection_db_isomorphism}
    Let $X$ be a normal variety over $\mathbb{C}$. Then the following are equivalent
    for any $m\geq 0$:
    \begin{enumerate}
    \item\label{item:1} $X$ has pre-$m$-rational singularities
    \item\label{item:2}
      $\mathcal{H}^0(\mathcal{I}\underline{\Omega}^p_X) \xrightarrow{ntrl.}
      \mathcal{I}\underline{\Omega}^p_X$ is an isomorphism for all $0\leq p \leq m$.
    \end{enumerate}
\end{lemma}

\begin{proof}
  From \cite[Lemma 9.4]{Popa/Shen/DucVo:2024}, $X$ having pre-$m$-rational
  singularities ensures that
  $\underline{\Omega}^p_X \xrightarrow{ntrl.} \mathcal{I}\underline{\Omega}^p_X$ is
  an isomorphism for all $0\leq p \leq m$. However, \cite[Theorem
  B]{Shen/Venkatesh/Vo:2023} (or \cite[Theorem 1.4]{Kovacs:2025}) tells us that $X$
  must also have pre-$m$-Du~Bois singularities. Therefore,
  $\mathcal{H}^0 (\underline{\Omega}^p_X) \xrightarrow{ntrl.} \underline{\Omega}^p_X$
  is an isomorphism for all $0\leq p \leq m$, and hence
  $\mathcal{H}^0(\mathcal{I}\underline{\Omega}^p_X) \xrightarrow{ntrl.}
  \mathcal{I}\underline{\Omega}^p_X$ is an isomorphism for all $0\leq p \leq m$.
    
  Conversely, assume that
  $\mathcal{H}^0(\mathcal{I}\underline{\Omega}^p_X) \xrightarrow{ntrl.}
  \mathcal{I}\underline{\Omega}^p_X$ is an isomorphism for all $0\leq p \leq m$. We
  will use induction on $m$ and start with the base case $m=0$. By \cite[(8.4.1) \&
  Proposition 8.2]{Kebekus/Schnell:2021}, we know that
  $\mathcal{I}\underline{\Omega}^0_X \xrightarrow{ntrl.} \sR f_\ast
  \mathcal{O}_{\widetilde{X}}$ is an isomorphism where $f\colon \widetilde{X}\to X$
  is a strong log resolution. We have
  $\mathcal{O}_X \simeq f_\ast \mathcal{O}_{\widetilde{X}}$ because $X$ is
  normal. Hence, \eqref{item:2} implies that
  $\mathcal{O}_X \xrightarrow{ntrl.} \sR f_\ast \mathcal{O}_{\widetilde{X}}$ is an
  isomorphism, and so $X$ has rational singularities.

  Next, assume that the desired claim holds for $0\leq p \leq m-1$, i.e., that $X$
  has pre-$(m-1)$-rational singularities. Then it also has pre-$(m-1)$-Du~Bois
  singularities (e.g.\ see \cite[Theorem B]{Shen/Venkatesh/Vo:2023} or \cite[Theorem
  1.4]{Kovacs:2025}). Furthermore,
  $\mathcal{H}^0 (\underline{\Omega}^m_X) \xrightarrow{ntrl.}
  \mathcal{H}^0(\mathcal{I}\underline{\Omega}^m_X)$ is an isomorphism by
  \Cref{lem:Kebekus_Schnell}, so the top and right arrows, and hence the diagonal
  arrow in the following diagram are isomorphisms,
  \[
    \xymatrix{%
      \mathcal{H}^0 (\underline{\Omega}^m_X) \ar[r]^-\simeq \ar[rd]^\simeq \ar[d] &
      \mathcal{H}^0(\mathcal{I}\underline{\Omega}^m_X) \ar[d]^\simeq \\
      \underline{\Omega}^m_X \ar[r] & \mathcal{I}\underline{\Omega}^m_X. }
  \]
  It follows that the natural morphism
  $\mathcal{H}^0 (\underline{\Omega}^m_X) \xrightarrow{ntrl.}
  \underline{\Omega}^m_X$ has a left inverse, and then $X$ has pre-$m$-Du~Bois
  singularities by \cite[Theorem 10.8]{Kovacs:2025}. Therefore,
  $\mathcal{H}^0 (\underline{\Omega}^m_X) \xrightarrow{ntrl.}
  \underline{\Omega}^m_X$ is actually an isomorphism, and hence
  $\underline{\Omega}^m_X \xrightarrow{ntrl.} \mathcal{I} \underline{\Omega}^m_X$
  must also be an isomorphism. Then $X$ has pre-$m$-rational singularities by
  \cite[Remark's 4.8(4) \& 5.2]{Dirks/Olano/Raychaudhury:2025}, which completes the
  proof.
\end{proof}

\begin{theorem}
    \label{thm:splitting_via_irr_complex}
    Let $X$ be a normal variety over $\mathbb{C}$. Then the following are equivalent
    for any $m\geq 0$:
    \begin{enumerate}
        \item \label{thm:splitting_via_irr_complex1} $X$ has pre-$m$-rational singularities
        \item \label{thm:splitting_via_irr_complex2}
          $\mathcal{H}^0(\mathcal{I}\underline{\Omega}^p_X) \xrightarrow{ntrl.}
          \mathcal{I}\underline{\Omega}^p_X$ has a left inverse for each
          $0\leq p \leq m$, and
        \item \label{thm:splitting_via_irr_complex3}
          $\widetilde{\Irr}^p_X \xrightarrow{ntrl.} \Irr^p_X$ has a left inverse for
          each $0\leq p \leq m$.
    \end{enumerate}
\end{theorem}

\begin{proof}
  From \Cref{lem:pre_m_rational_via_intersection_db_isomorphism}, we know that
  $\eqref{thm:splitting_via_irr_complex1}\implies
  \eqref{thm:splitting_via_irr_complex2}$, whereas
  $\eqref{thm:splitting_via_irr_complex1} \implies
  \eqref{thm:splitting_via_irr_complex3}$ is straightforward by the definition of
  pre-$m$-rational singularities.

  Next, we show that
  $\eqref{thm:splitting_via_irr_complex2}\implies
  \eqref{thm:splitting_via_irr_complex1}$ by induction on $m$.
    
  By \cite[8.4.1 \& Proposition 8.2]{Kebekus/Schnell:2021}, we know that
  $\mathcal{I}\underline{\Omega}^0_X \xrightarrow{ntrl.} \sR f_\ast
  \mathcal{O}_{\widetilde{X}}$ is an isomorphism where $f\colon \widetilde{X}\to X$
  is a strong log resolution. We have
  $\mathcal{O}_X \simeq f_\ast \mathcal{O}_{\widetilde{X}}$ because $X$ is
  normal. Hence, our hypothesis implies that
  $\mathcal{O}_X \xrightarrow{ntrl.} \sR f_\ast \mathcal{O}_{\widetilde{X}}$ splits,
  and so $X$ has rational singularities by \cite[Theorem 1]{Kovacs:2000}.
    
  Next assume that $m\geq 1$ and choose a $p$, such that $0\leq p \leq m$. Then
  $\mathcal{H}^0( \mathcal{I}\underline{\Omega}^p_X) \xrightarrow{ntrl.}
  \mathcal{I}\underline{\Omega}^p_X$ admits a left inverse. Applying the functor
  $\sR\sHom(- , \omega^\kdot_X)$ gives a composition
    \begin{displaymath}
      \sR\sHom( \mathcal{H}^0 (\mathcal{I}\underline{\Omega}^p_X ), \omega^\kdot_X)
      \to \sR\sHom( \mathcal{I}\underline{\Omega}^p_X, \omega^\kdot_X) \to
      \sR\sHom(\mathcal{H}^0 (\mathcal{I}\underline{\Omega}^p_X ) , \omega^\kdot_X) 
    \end{displaymath}
    which is an isomorphism (in the derived category). In particular, we see that
    \begin{equation}
        \label{eq:pre_k_rational_dualized1}
        \sR\sHom( \mathcal{I}\underline{\Omega}^p_X, \omega^\kdot_X) \to
        \sR\sHom(\mathcal{H}^0 (\mathcal{I}\underline{\Omega}^p_X ) , \omega^\kdot_X) 
    \end{equation}
    is surjective on cohomology sheaves. However, \cite[Conjecture
    10.1]{Popa/Shen/DucVo:2024} coupled with \cite[Theorem 1.1]{Kovacs:2025} and
    \cite[Theorem 10.3]{Popa/Shen/DucVo:2024}, ensures that
    \eqref{eq:pre_k_rational_dualized1} must also be injective on cohomology sheaves
    (cf.~the comments after \cite[Corollary 1.3]{Kovacs:2025}).  It follows that
    \eqref{eq:pre_k_rational_dualized1} is an isomorphism and hence, after dualizing
    once more, we conclude that
    $\mathcal{H}^0(\mathcal{I}\underline{\Omega}^p_X) \xrightarrow{ntrl.}
    \mathcal{I}\underline{\Omega}^p_X$ is an isomorphism. Then
    \Cref{lem:pre_m_rational_via_intersection_db_isomorphism} implies that $X$ has
    pre-$m$-rational singularities, which completes the proof of this implication.

    Lastly we need to check that
    $\eqref{thm:splitting_via_irr_complex3} \implies
    \eqref{thm:splitting_via_irr_complex2}$. Let $0\leq p \leq m$ and consider the diagram,
    \[
      \xymatrix{%
        {\mathcal{H}^0 (\mathcal{I}\underline{\Omega}^p_X)} \ar[r]^-\simeq \ar[d] &
        {\widetilde{\Irr}^p_X} \ar[d] \\
        {\mathcal{I}\underline{\Omega}^p_X} \ar[r] & {\Irr^p_X.}
        \ar@/_1em/[u]_{\text{left inverse}} }
    \]
    Here, the top arrow is an isomorphism by \Cref{lem:Kebekus_Schnell} and the right 
    arrow admits a left inverse by \eqref{thm:splitting_via_irr_complex3}.  It follows
    that
    $\mathcal{H}^0(\mathcal{I}\underline{\Omega}^p_X) \xrightarrow{ntrl.}
    \mathcal{I}\underline{\Omega}^p_X $ has a left inverse, which completes the
    proof.
\end{proof}

\begin{corollary}
    \label{cor:finite_cover_pre_m_rational}
    Let $f\colon Y \to X$ be a finite surjective morphism of normal varieties. If $Y$
    has pre-$m$-rational singularities, then so does $X$.
\end{corollary}

\begin{proof}
  First, we will show that there is a natural morphism of Hodge modules
  $\mathcal{IC}_X^H \to f_\ast\mathcal{IC}_Y^H$. We have the following natural
  morphism in $D^b(\operatorname{MHM}(X))$
  \begin{displaymath}
    \mathbb{Q}^H_X [\dim X] \to f_\ast \mathbb{Q}^H_Y [\dim Y].
  \end{displaymath}
  As $f$ is finite, $f_\ast\colon D^b(\operatorname{MHM}(Y)) \to D^b(\operatorname{MHM}(X))$ is $t$-exact, i.e., if $M \in \operatorname{MHM}(Y)$, then $f_\ast M \in \operatorname{MHM}(X)$. So, we get the following by taking $\mathcal{H}^0$ on both sides
  \begin{displaymath}
    \mathcal{H}^0(\mathbb{Q}^H_X [\dim X]) \to f_\ast \mathcal{H}^0(\mathbb{Q}^H_Y
    [\dim Y]). 
  \end{displaymath}
  If we take $\operatorname{gr}^W_{\dim X}$ on both sides, this commutes with
  $f_\ast$ because of the weight spectral sequence (see \cite[Theorem
  4.13]{Kebekus/Schnell:2021} or \cite[Proposition 2.15]{Saito:1990}). Thus, we get a
  natural map
  \begin{displaymath}
    \mathcal{IC}_X^H \simeq \operatorname{gr}^W_{\dim X}
    (\mathcal{H}^0(\mathbb{Q}^H_X [\dim X])) \to f_\ast \operatorname{gr}^W_{\dim
      X}(\mathcal{H}^0(\mathbb{Q}^H_Y [\dim Y])) \simeq f_\ast\mathcal{IC}_Y^H, 
  \end{displaymath}
  which fits in the following commutative diagram
  \begin{displaymath}
    \begin{tikzcd}
      {\mathbb{Q}^H_X [\dim X]} & {f_\ast \mathbb{Q}^H_Y [\dim Y]} \\
      {\mathcal{IC}_X^H} & {f_\ast \mathcal{IC}_Y^H.}
      \arrow[from=1-1, to=1-2]
      \arrow[from=1-1, to=2-1]
      \arrow[from=1-2, to=2-2]
      \arrow[from=2-1, to=2-2]
    \end{tikzcd}
  \end{displaymath}
  Now, for any $p \in \mathbb{Z}$, we apply
  $\operatorname{gr}^F_{-p} \operatorname{DR}_X (\cdot)[p- \dim X]$ to all the
  objects in the diagram and use the fact that
  $\operatorname{gr}^F_{-p} \operatorname{DR}$ commutes with pushforward for proper maps (see
  \cite[\S 2.3.7]{Saito:1988}). As $f$ is finite, $f_\ast=\sR f_*$, so we get the commuting diagram
  \begin{displaymath}
    \begin{tikzcd}
      {\underline{\Omega}^p_X} & {
        f_\ast \underline{\Omega}^p_Y} \\
      {\mathcal{I}\underline{\Omega}^p_X} & {
        f_\ast
        \mathcal{I}\underline{\Omega}^p_Y.} 
      \arrow["{\phi^p}", from=1-1, to=1-2]
      \arrow["{\psi_X^p}"', from=1-1, to=2-1]
      \arrow["{\psi_Y^p}", from=1-2, to=2-2]
      \arrow["{\mathcal{I}\phi^p}", from=2-1, to=2-2]
    \end{tikzcd}
  \end{displaymath}
  We now prove that $X$ is pre-$m$-rational closely following the strategy of
  \cite[Proposition 4.2(1)]{Shen/Venkatesh/Vo:2023}. Observe first that $Y$ has
  rational singularities and so $X$ does as well by \cite[Theorem
  1]{Kovacs:2000}. Since $Y$ is pre-$m$-rational, for any $p \leq m$, both $
  f_\ast \underline{\Omega}^p_Y$ and $
  f_\ast \mathcal{I}\underline{\Omega}^p_Y$ are sheaves, and $\psi^p_Y$ is an
  isomorphism by \Cref{lem:pre_m_rational_via_intersection_db_isomorphism}. $X$ is
  pre-$m$-Du~Bois by \cite[Proposition 4.2(1)]{Shen/Venkatesh/Vo:2023}, and so
  $\underline{\Omega}^p_X \simeq
  \mathcal{H}^0(\underline{\Omega}^p_X)$. Additionally, since $X$ and $Y$ have
  rational singularities, we have
  $\mathcal{H}^0(\underline{\Omega}^p_X) \simeq
  \mathcal{H}^0(\mathcal{I}\underline{\Omega}^p_X) \simeq \Omega^{[p]}_X$ and $
  f_\ast \underline{\Omega}^p_Y \simeq f_\ast \Omega^{[p]}_Y$ by
  \Cref{lem:Kebekus_Schnell} and \cite[Corollary~3.18]{Kovacs:2025}. Therefore,
  exactly as in the proof of \cite[Proposition 4.2(1)]{Shen/Venkatesh/Vo:2023},
  $\phi^p$ admits a left inverse by \cite{Zannier:1999}, which we denote by
  $\xi^p$. Then $\xi^p \circ (\psi^p_Y)^{-1} \circ \mathcal{I}\phi^p$ is a left
  inverse of $\psi^p_X$, and then \Cref{thm:splitting_via_irr_complex} implies that
  $X$ is pre-$m$-rational.
  \end{proof}

\bibliographystyle{skalpha}
\bibliography{mainbib}

\end{document}